\documentclass[12pt]{article}
\usepackage{amsfonts, amsmath, amssymb, latexsym, eucal, amscd} 

\usepackage{amsthm}
\usepackage{amscd}

\usepackage{cite}

\newtheorem{definition}{Definition}[section]
\newtheorem{theorem}[definition]{Theorem}
\newtheorem{lemma}[definition]{Lemma}
\newtheorem{corollary}[definition]{Corollary}

\newtheorem{example}[definition]{Example}

\newtheorem{problem}[definition]{Problem}
\newtheorem{note}[definition]{Note}
\newtheorem{assumption}[definition]{Assumption}
\newtheorem{proposition}[definition]{Proposition}

\begin{document} 

\title{\bf The Norton algebra of a $Q$-polynomial 
distance-regular graph
}
\author{
Paul Terwilliger}
\date{}

\maketitle
\begin{abstract} We consider the Norton algebra associated with a $Q$-polynomial primitive idempotent of the adjacency matrix for a distance-regular graph.
We obtain a formula for the Norton algebra product that we find attractive.

\bigskip

 \end{abstract}

\section{Introduction} There is a family of highly regular graphs said to be 
distance-regular \cite{banIto, BCN, dkt}. 
%We mention some examples. For each of the five platonic solids the 1-skeleton is a distance-regular graph. 
Examples include 
the Johnson graphs \cite[Section~9.1]{BCN},
the Hamming graphs \cite[Section 9.2]{BCN}, 
the Grassmann graphs \cite[Section~9.3]{BCN},
and the dual polar graphs \cite[Section~9.4]{BCN}. The graphs in these four families are particularly
attractive for several reasons: (i) they have a $Q$-polynomial structure, according to which 
their Krein parameters vanish in a certain attractive pattern; (ii) these graphs come with
 a ranked partially ordered set that can be used to analyze the graph. 
 \medskip
 
 \noindent In the analysis of any distance-regular graph $\Gamma$, one often considers 
 the eigenspaces for the adjacency matrix $A$ of $\Gamma$.  By \cite{norton}  these eigenspaces
 possess an algebra structure, called the Norton algebra, that is commutative but not necessarily associative.
 The Norton product  $\star$ is described as follows. 
Let $X$ denote the vertex set of $\Gamma$.
%%%%and let $A$ denote the adjacency matrix (See Section 2 for formal definitions).
The  rows and columns of $A$ are indexed by $X$. The matrix $A$ acts on
 a vector space $V$ over $\mathbb R$, consisting of column vectors whose coordinates are indexed by $X$.
 For $x \in X$ let $\hat x$ denote the vector in $V$ that has $x$-coordinate 1 and all other coordinates zero.
 So $\lbrace \hat x | x \in X\rbrace$ is a basis for $V$.  The entry-wise product $\circ: V \times V \to V$ satisfies
 $\hat x \circ \hat y = \delta_{x,y} \hat x$ for all $x,y \in X$. The matrix $A$ is diagonalizable since it is symmetric, 
 so $V$ is a direct sum of the eigenspaces of $A$.
 For an eigenspace of $A$, the corresponding
 primitive idempotent $E$  acts as the identity on the eigenspace, and zero on the other eigenspaces of $A$. Thus $E$ is the projection
 from $V$ onto the eigenspace. The  eigenspace is $EV$.
For $u, v \in EV$ we have
 $u \star v = E (u \circ v) $.
 \medskip
 
 \noindent
  Earlier we mentioned $Q$-polynomial structures.  For a given $Q$-polynomial structure on $\Gamma$
  the adjacency matrix $A$ has 
  a distinguished primitive idempotent, said to be $Q$-polynomial. 
  Recently, several authors have considered  the Norton algebra $EV$ for a $Q$-polynomial primitive idempotent $E$ of $A$.
  This was done by 
 C. Maldonado and D. Penazzi in \cite{Norton},  under the assumption that $\Gamma$ is 
 a Johnson graph, Hamming graph, or  Grassmann graph.
 It was done by 
 F. Levstein, C. Maldonado, and D. Penazzi in \cite{levNorton}, 
under the assumption that $\Gamma$ is a dual polar graph.  In both articles the authors compute
$\check u \star  \check v $  for all $u,v \in L_1$, where
$\lbrace {\check u} \vert u \in L_1\rbrace$ is a certain spanning set for $EV$ indexed by the set $L_1$ of rank 1 elements in the associated poset.
The results of \cite{Norton, levNorton} are used by J. Huang
 in  \cite{jia} to investigate the extent to which the Norton product is nonassociative.
 \medskip

\noindent In the present paper we consider the Norton algebra $EV$, where $E$ is a  $Q$-polynomial primitive idempotent of the adjacency matrix $A$ for any distance-regular graph $\Gamma$ with diameter $d\geq 2$.
For all vertices $x,y$ of $\Gamma$ we give an explicit formula for the Norton product $E \hat x \star E\hat y$, in terms of a few eigenvalues $\theta_0, \theta_1, \theta_2$
of $A$
and a sequence of scalars $\lbrace \theta^*_i \rbrace_{i=0}^d$  called the dual eigenvalues of $\Gamma$ associated with $E$.
We give two versions of our formula. The first version is more straightforward. To obtain the second version, we use the balanced set condition
\cite{QPchar, newineq}
to make the symmetry $E\hat x \star E \hat y = E \hat y \star E \hat x$ explicit.
 Our main results are Theorems \ref{thm:main} and  \ref{thm:main2}.
 \medskip
 
 \noindent The paper is organized as follows.
  In Section 2 we review some basic concepts concerning distance-regular graphs.
 In Section 3 we recall the Norton algebra and obtain the first version of our Norton product formula.
 In Section 4 we use the balanced set condition to obtain the second version of our Norton product formula.
 In Section 5 we remark how certain equations in Sections 3, 4 can be obtained using the theory of Leonard systems.

\section{Preliminaries} In this section we review some basic concepts concerning distance-regular graphs. For more background
information we refer the reader to
\cite{banIto, BCN, dkt}.
\medskip

\noindent
Let $\mathbb R$ denote the real number field.
Let $X$ denote a nonempty finite set.
Let ${\rm Mat}_X(\mathbb R)$ denote the $\mathbb R$-algebra consisting of the matrices that have rows and columns indexed by $X$ and all entries in $\mathbb R$.
Let $I$ denote the identity matrix in  ${\rm Mat}_X(\mathbb R)$,
and let $J$ denote the matrix in  ${\rm Mat}_X(\mathbb R)$ that has all entries 1.
 Let $V=\mathbb R X$ denote the vector space over $\mathbb R$ consisting
of the column vectors that have coordinates indexed by $X$ and all entries in $\mathbb R$. The algebra
${\rm Mat}_X(\mathbb R)$ acts on $V$ by left multiplication.
%%We call $V$ the {\it standard module}.
For $x \in X$  let $\hat x$ denote the vector in $V$ that has $x$-coordinate 1 and all other coordinates 0.
The vectors $\lbrace \hat x | x \in X\rbrace$ form a basis for $V$. 
Let $\bf 1$ denote the vector in $V$ that has all entries 1. So ${\bf 1} = \sum_{x \in X} \hat x$.
Note that $J \hat x = \bf 1$ for all $x \in X$.
\medskip

\noindent Let $\Gamma=(X,R)$ denote an undirected connected graph, without loops or multiple edges, with vertex set $X$,
edge set $R$, and path-length distance function $\partial$. Recall the diameter $d={\rm max}\lbrace \partial (x,y) \vert x, y \in X\rbrace$.
For $x \in X$ and an integer $i\geq 0$ define $\Gamma_i(x) = \lbrace y \in X \vert \partial(x,y)=i\rbrace$.  We abbreviate
$\Gamma(x) = \Gamma_1(x)$.
For an integer $k\geq 0$, $\Gamma$ is said to be {\it regular with valency $k$} whenever 
$k=\vert \Gamma(x)\vert$ for $x \in X$.
The graph $\Gamma$ is said to be {\it distance-regular} whenever for all integers $h,i,j$ $(0 \leq h,i,j\leq d)$
and all $x,y\in X$ at distance $\partial(x,y)=h$, the scalar
%%\begin{align*}
$ p^h_{i,j} = \vert \Gamma_i(x) \cap \Gamma_j(y)\vert$
%%\end{align*}
is independent of $x$ and $y$. The scalars $p^h_{i,j} $ are called the {\it intersection numbers} of $\Gamma$.
For the rest of this paper, we assume that $\Gamma$ is distance-regular with diameter $d\geq 2$. 
By construction $p^h_{i,j}=p^h_{j,i}$ for $0 \leq h,i,j\leq d$.
By the triangle inequality we find that for $0 \leq h,i,j\leq d$,
\begin{enumerate}
\item[\rm (i)] $p^h_{i,j}=0$ if one of $h,i,j$ is greater than the sum of the other two;
\item[\rm (ii)] $p^h_{i,j}\not=0$ if one of $h,i,j$ is equal to the sum of the other two.
\end{enumerate}
We abbreviate
$c_i = p^i_{1,i-1}$ $(1 \leq i \leq d)$,
$a_i = p^i_{1,i}$ $(0 \leq i \leq d)$,
$b_i = p^i_{1,i+1}$ $(0 \leq i \leq d-1)$.
By construction $c_i \not=0$ for $1 \leq i \leq d$ and $b_i \not=0$ for $0 \leq i \leq d-1$.
The graph $\Gamma$ is regular with valency $k=b_0$. Moreover
$k = c_i + a_i + b_i$ for $0 \leq i \leq d$, 
where $c_0=0$ and $b_d=0$.
\medskip

\noindent Next we recall the Bose-Mesner algebra of $\Gamma$.
For $0 \leq i \leq d$ let $A_i$ denote the matrix in ${\rm Mat}(\mathbb R)$  that has
$(x,y)$-entry
\begin{align*}
(A_i)_{x,y} = \begin{cases}
 1, & {\mbox{\rm if $\partial(x,y)=i$}}; \\
0, &{\mbox{\rm if $\partial(x,y)\not=i$}}
\end{cases}
\qquad \qquad (x,y \in X).
\end{align*}
We call $A_i$ the {\it $i^{\rm th}$ distance-matrix} of $\Gamma$.
Note that $A_0=I$. We abbreviate $A=A_1$ and call this the {\it adjacency matrix} of $\Gamma$.
By the construction,
\begin{align*}
A_i A_j = \sum_{h=0}^d p^h_{i,j} A_h \qquad \qquad (0 \leq i,j\leq d).
\end{align*}
By these comments the matrices $\lbrace A_i \rbrace_{i=0}^d$ form a basis for a commutative subalgebra of ${\rm Mat}_X(\mathbb R)$.
This algebra is denoted by $M$ and called the {\it Bose-Mesner algebra} of $\Gamma$. The algebra $M$ is generated by $A$ \cite[p.~190]{banIto}.
\medskip

\noindent Next we recall the primitive idempotents and eigenvalues of $\Gamma$.
By \cite[p.~45]{BCN} the vector space $M$ has a basis $\lbrace E_i \rbrace_{i=0}^d$ such that (i) $E_0=\vert X \vert^{-1} J$;
%\qquad \quad 
(ii) $I = \sum_{i=0}^d E_i$;
(iii) $E_i E_j = \delta_{i,j} E_i $ $ (0 \leq i,j\leq d)$.
%\end{align*}
This basis is unique up to permutation of $\lbrace E_i \rbrace_{i=1}^d $.
We call $\lbrace E_i \rbrace_{i=0}^d$ the {\it primitive idempotents} of $M$ (or $\Gamma$).
The primitive idempotent $E_0$ is called {\it trivial}.
By construction, there exist real numbers $\lbrace \theta_i\rbrace_{i=0}^d$ such that
$A = \sum_{i=0}^d \theta_i E_i$.
The $\lbrace \theta_i\rbrace_{i=0}^d$ are mutually distinct since $A$ generates $M$. 
Using $E_0=\vert X \vert^{-1} J$ we obtain $\theta_0=k$. 
The scalars
$\lbrace \theta_i\rbrace_{i=0}^d$ are called the {\it eigenvalues} of $A$ (or $\Gamma$).
\medskip

\noindent Next we recall the Krein parameters of $\Gamma$.
 For $0 \leq i,j\leq d$ we have $A_i \cdot A_j = \delta_{i,j}A_i$, where $\cdot$ denotes
the entry-wise product for ${\rm Mat}_X(\mathbb R)$. Therefore $M$ is closed under the $\cdot$ product. Consequently there exist $q^h_{i,j} \in \mathbb R$ $(0 \leq h,i,j\leq d)$ such that
\begin{align}
\label{eq:eee}
E_i \cdot E_j = \vert X \vert^{-1} \sum_{h=0}^d q^h_{i,j} E_h \qquad \qquad (0 \leq i,j\leq d).
\end{align}
By construction $q^h_{i,j} = q^h_{j,i} $ for $0 \leq h,i,j\leq d$.
By \cite[Proposition~4.1.5]{BCN} we have $q^h_{i,j}\geq 0$ for $0 \leq h,i,j\leq d$. The scalars $q^h_{i,j}$ are called the {\it Krein parameters} of $\Gamma$.
\medskip

\noindent We describe one significance of the Krein parameters. In this description, we will use the following notation.
For $u \in V$ and $x \in X$ let $u_x$ denote the $x$-coordinate of $u$. So $u = \sum_{x \in X} u_x \hat x$.
 For $u, v \in V$ their entry-wise product $u\circ v$ is the vector in $V$ that has $x$-coordinate $u_x v_x$ for all
$x \in X$. So
%%\begin{align*}
$u \circ v = \sum_{x \in X} u_x v_x \hat x$.
%%\end{align*}
For $x, y \in X$ we have
\begin{align}
\label{eq:xycirc}
\hat x \circ \hat y = \begin{cases}
 \hat x, & {\mbox{\rm if $x=y$}}; \\
0, &{\mbox{\rm if $x\not=y$.}}
\end{cases}
\end{align} 
For $v \in V$ we have ${\bf 1} \circ v = v$.  For subspaces $Y, Z $ of $V$ define
$Y \circ Z = {\rm Span} \lbrace y \circ z \vert y \in Y, z \in Z\rbrace$.
 By \cite[Proposition~5.1]{norton} we have
\begin{align}\label{prop:norton}
E_i V \circ E_j V = \sum_{\stackrel{ \scriptstyle 0 \leq h \leq d }{ \scriptstyle q^h_{ij} \not=0}} E_hV
\qquad \qquad (0 \leq i,j\leq d).
\end{align}
%%%%%%%%%%%

\noindent Next we recall the $Q$-polynomial property.
The given ordering $\lbrace E_i \rbrace_{i=1}^d$ of the nontrivial primitive idempotents of $\Gamma$ is said to be {\it $Q$-polynomial} whenever
for $0 \leq h,i,j\leq d$,
\begin{enumerate}
\item[\rm (i)] $q^h_{i,j}=0$ if one of $h,i,j$ is greater than the sum of the other two;
\item[\rm (ii)] $q^h_{i,j}\not=0$ if one of $h,i,j$ is equal to the sum of the other two.
\end{enumerate}
Let $E$ denote a nontrivial primitive idempotent of $\Gamma$. We say that $E$  is {\it  $Q$-polynomial}
whenever there exists a $Q$-polynomial ordering $\lbrace E_i\rbrace_{i=1}^d$ of the nontrivial primitive idempotents of $\Gamma$
such that $E=E_1$. For the rest of this paper we assume that $E$ is $Q$-polynomial.
By construction, there exist real numbers $\lbrace \theta^*_i\rbrace_{i=0}^d$  such that
\begin{align}
\label{eq:Esum}
E = \vert X \vert^{-1} \sum_{i=0}^d \theta^*_i A_i.
\end{align}
By \cite[p.~260]{banIto} the scalars $\lbrace \theta^*_i \rbrace_{i=0}^d$ are mutually distinct.
The scalars $\lbrace \theta^*_i \rbrace_{i=0}^d$ 
are called the {\it dual eigenvalues} of $\Gamma$ associated with $E$. For notational convenience let $\theta^*_{-1}$ and $\theta^*_{d+1}$
denote indeterminates.  Taking the trace of each side of (\ref{eq:Esum}) yields $\theta^*_0 = {\rm rank}(E)$. Also,
multiplying both sides of (\ref{eq:Esum}) by $A$ and evaluating the result yields
\begin{align}
c_i \theta^*_{i-1} +
a_i \theta^*_i +
b_i \theta^*_{i+1} 
= \theta_1 \theta^*_i
\qquad \qquad (0 \leq i \leq d).
\label{eq:recurse}
\end{align}

\section{The Norton algebra}
\noindent We continue to discuss the distance-regular graph $\Gamma$ and its  $Q$-polynomial primitive idempotent $E$.
In this section we turn the vector space $EV$ into a commutative nonassociative algebra called the Norton algebra.
%We will define this algebra shortly.
%\medskip

\begin{definition} {\rm (See \cite[Proposition~5.2]{norton}.)}
\label{def:norton} \rm
The {\it Norton algebra} of $\Gamma$ consists of 
the vector space $EV$, together with the product
\begin{align*}
% u \ast  v := E (u \circ v) \qquad \qquad (u,v \in EV).
 u \star v = E (u \circ v) \qquad \qquad (u,v \in EV).
 \end{align*}
 \end{definition}
\medskip

\noindent The Norton algebra is commutative, but not necessarily associative.
\medskip

\noindent The vector space $EV$ is spanned by the vectors $\lbrace E\hat x \vert x \in X\rbrace$.
These vectors are nonzero, mutually distinct, and
 linearly dependent \cite[Theorem~1.1]{QPchar}.
As we investigate
the Norton product $\star$ it is natural to consider $E\hat x \star E\hat y$ for all $x,y \in X$.
In the next two lemmas we discuss some extremal cases.

\begin{lemma} \label{lem:warm}
For $x \in X$,
\begin{align}
E \hat x \star E \hat x = |X|^{-1} q^1_{1,1}E \hat x.
\label{eq:ese}
\end{align}
\end{lemma}
\begin{proof} By (\ref{eq:eee}) we have $E \cdot E = \vert X \vert^{-1} \sum_{h=0}^d q^h_{1,1} E_h$.  For this equation, multiply each side by $E$ to obtain
$E (E \cdot E) = \vert X \vert^{-1} q^1_{1,1} E$. For this equation, compare column $x$  of each side to obtain (\ref{eq:ese}).
\end{proof}

\begin{lemma} \label{lem:warmup1}
The following are equivalent:
\begin{enumerate}
\item[\rm (i)] $E\hat x \star E \hat y =0$ for all $x,y\in X$;
\item[\rm (ii)] $u \star v = 0$ for all $u,v \in EV$;
\item[\rm (iii)] The Krein parameter $q^1_{1,1} =0$.
\end{enumerate}
\end{lemma}
\begin{proof} By 
(\ref{prop:norton}) and the construction.
\end{proof}

\noindent We have been discussing some extremal cases. Before we proceed to the general case, we bring in some notation.
To motivate things, observe that for $x \in X$ and $0 \leq i \leq d$,
\begin{align}
A_i \hat x = \sum_{z \in \Gamma_i(x)}  {\hat z}.
%%A_i \hat x = \sum_{\scriptstyle z \in X \atop \scriptstyle \partial (z,x)= i}  {\hat z}.
\label{eq:AiExp}
\end{align}
\begin{lemma}\label{lem:double} For
$x,y \in X$ and $0 \leq i,j\leq d$ we have
\begin{align*}
A_i \hat x \circ A_j \hat y = 
\sum_{z \in \Gamma_i(x) \cap \Gamma_j(y)}  {\hat z}.
 \end{align*}
 \end{lemma}
 \begin{proof} Use
 (\ref{eq:xycirc})
  and
 (\ref{eq:AiExp}).
 \end{proof}
 
\begin{definition}\label{def:xy}
\rm Pick $x, y \in X$ and write $i = \partial (x,y)$. Define
\begin{align}
 x^+_y &= A \hat x \circ A_{i+1} \hat y = \sum_{z \in \Gamma(x) \cap \Gamma_{i+1}(y)}  {\hat z},
 \label{eq:xyp}
\\
x^0_y &= A \hat x \circ A_{i} \hat y = 
\sum_{z\in \Gamma(x) \cap \Gamma_i(y)} {\hat z},
\label{eq:xy0}
\\x^-_y &= A \hat x \circ A_{i-1} \hat y = 
\sum_{z \in \Gamma(x) \cap \Gamma_{i-1}(y) }  {\hat z},
\label{eq:xym}
 \end{align}
\noindent where we understand
$A_{-1}=0$, $\Gamma_{-1}(x)= \emptyset$ and
$A_{d+1}=0$, $\Gamma_{d+1}(x)=\emptyset$.
%%$A_{j}=0$ and $\Gamma_{j}(x) = \emptyset $ for $j\in \lbrace -1, d+1\rbrace $.
%where $\partial (x,y) = i$. %Thus $x^+_y = 0$ if $i=D$. Moreover  $x^-_y=0$ and $x^0_y = 0$ if $i=0$.
\end{definition}
\noindent We clarify the notation (\ref{eq:xyp})--(\ref{eq:xym}).
 Pick $x, y \in X$. If $\partial(x,y)=d$ then $x^+_y = 0$. If $\partial(x,y)=1$ then $x^-_y = \hat y$.  If $x=y$ then $x^0_y=0$ and $x^-_y=0$.

\begin{lemma}
\label{lem:sum} For $x,y \in X$ we have
\begin{align}
x^+_y + x^0_y + x^-_y &= A \hat x,
\label{eq:sum1}
\\
Ex^+_y + Ex^0_y + Ex^-_y &= \theta_1 E \hat x.
\label{eq:sum2}
\end{align}
\end{lemma}
\begin{proof} To verify (\ref{eq:sum1}), note that each side is equal to 
$\sum_{z \in \Gamma(x)} \hat z$.
To get (\ref{eq:sum2}), apply $E$ to each side of (\ref{eq:sum1}),  and use $EA = \theta_1 E$.
\end{proof}
 
 \noindent The following is our first main result.

\begin{theorem}\label{thm:main}
Assume that $\Gamma$ is  $Q$-polynomial with respect to $E$.
Then for all $x,y\in X$ we have
\begin{align}
 E{\hat x}\star E{\hat y} &=
 \frac{ (\theta^*_{i-1}-\theta^*_i) Ex^-_y+(\theta^*_{i+1}-\theta^*_{i})
Ex_y^+ +(\theta_1-\theta_2)\theta^*_i E{\hat x}+(\theta_2-\theta_0)E{\hat y}}
  {\vert X \vert (\theta_1-\theta_2)}
  \label{eq:main}
\end{align}
where $i = \partial (x,y)$. Here $\theta^*_{-1} $ and $\theta^*_{d+1} $ denote indeterminates.
\end{theorem}
\begin{proof} We consider the vector
\begin{align*}
 E (A \hat x \circ E \hat y) -\theta_2 
 E (\hat x \circ E \hat y).
 %\label{eq:start}
 \end{align*}
We evaluate this vector in two ways.
For the first evaluation, use $A-\theta_2 I = \sum_{h=0}^d (\theta_h - \theta_2) E_h$
to obtain
\begin{align*}
 E (A \hat x \circ E \hat y) -\theta_2 
 E (\hat x \circ E \hat y) 
&=\sum_{h=0}^d (\theta_h-\theta_2) E (E_h \hat x \circ E \hat y).
\end{align*}
For the above sum, we examine the $h$-summand for $0 \leq  h \leq d$.  For $h=0$ the summand is $(\theta_0-\theta_2)\vert X \vert^{-1} E \hat y$ because
\begin{align*}
E_0 \hat x \circ E \hat y = \vert X \vert^{-1} J \hat x \circ E \hat y = \vert X \vert^{-1} {\bf 1} \circ E \hat y = \vert X \vert^{-1} E \hat y.
\end{align*}
For $h=1$ the summand is
$(\theta_1-\theta_2)E\hat x \star E \hat y$ by Definition \ref{def:norton}.
For  $h=2$ the summand is zero by construction. For 
$3 \leq h \leq d$ the summand is zero by (\ref{prop:norton}) and the definition of $Q$-polynomial.
 By these comments,
\begin{align}
 E (A \hat x \circ E \hat y) -\theta_2 
 E (\hat x \circ E \hat y) =
 (\theta_0-\theta_2) \vert X \vert^{-1} E \hat y + (\theta_1-\theta_2)E \hat x \star E \hat y.
\label{eq:L1}
\end{align}
For the second evaluation, use $E= \vert X \vert^{-1} \sum_{\ell=0}^d \theta^*_\ell A_\ell$ 
to obtain
\begin{align*}
(A-\theta_2 I) \hat x \circ E \hat y &= \vert X \vert^{-1} \sum_{\ell=0}^d \theta^*_\ell (A-\theta_2 I) \hat x \circ A_\ell \hat y.
\end{align*}
For the above sum, we examine the $\ell$-summand for $0 \leq \ell \leq d$.
The term $A \hat x \circ A_\ell \hat y$ is equal to $x^-_y$ (if $\ell=i-1$) and
$x^0_y$ (if $\ell=i$) and $x^+_y$ (if $\ell=i+1$) and zero 
(if $\vert \ell-i\vert > 1$). The term $\hat x \circ A_\ell \hat y$ is equal to
$\hat x$ (if $\ell=i$) and zero (if $\ell\not=i$). By these comments,
\begin{align*}
(A-\theta_2 I) \hat x \circ E \hat y &= \vert X \vert^{-1}(\theta^*_{i-1} x^-_y +
\theta^*_i x^0_y + \theta^*_{i+1} x^+_y - \theta_2 \theta^*_i  \hat x).
\end{align*}
Therefore
\begin{align}
 E (A \hat x \circ E \hat y) - \theta_2 E(\hat x \circ E \hat y) 
&= \vert X \vert^{-1}(\theta^*_{i-1} E x^-_y +
\theta^*_i E x^0_y + \theta^*_{i+1} E x^+_y - \theta_2 \theta^*_i E\hat x).
\label{eq:L2}
\end{align}
Comparing (\ref{eq:L1}), (\ref{eq:L2}) we obtain
\begin{align}
\vert X \vert  (\theta_1-\theta_2)E \hat x \star E \hat y
=
\theta^*_{i-1} E x^-_y +
\theta^*_i E x^0_y + \theta^*_{i+1} E x^+_y - \theta_2 \theta^*_i E\hat x 
+(\theta_2-\theta_0) E \hat y.
\label{eq:two}
\end{align}
In (\ref{eq:two}), eliminate the term $E x^0_y$  using (\ref{eq:sum2}).
%Lemma \ref{lem:sum}(ii).
In the resulting equation, solve for $E \hat x \star E \hat y$ and we are done.
\end{proof}

\noindent Referring to Theorem \ref{thm:main},  the formula for $E \hat x \star E \hat y$ can be simplified if $i\in \lbrace 0,1,d\rbrace $.  This simplification is discussed next.
\begin{corollary} \label{cor:01D} 
Assume that $\Gamma$ is  $Q$-polynomial with respect to $E$. Then {\rm (i)--(iii)} hold below.
\begin{enumerate}
\item[\rm (i)] For $x \in X$,
\begin{align*}
E \hat x \star E \hat x = \frac{\theta_1 \theta^*_1 - \theta_2 \theta^*_0 + \theta_2 - \theta_0}{\vert X \vert (\theta_1 - \theta_2)} E\hat x.
\end{align*}
\item[\rm (ii)] For $x,y \in X$ at distance $\partial(x,y)=1$,
\begin{align*}
E\hat x \star E \hat y = 
\frac{ (\theta^*_2-\theta^*_1) E x^+_y +(\theta_1-\theta_2)\theta^*_1 E \hat x + (\theta_2-\theta_0+\theta^*_0 - \theta^*_1)E \hat y}{\vert X \vert (\theta_1-\theta_2)}.
\end{align*}
\item[\rm (iii)] For $x,y \in X$ at distance $\partial(x,y)=d$,
\begin{align*}
 E{\hat x}\star E{\hat y} &=
 \frac{ (\theta^*_{d-1}-\theta^*_d) Ex^-_y +(\theta_1-\theta_2)\theta^*_d E{\hat x}+(\theta_2-\theta_0)E{\hat y}}
  {\vert X \vert (\theta_1-\theta_2)}.
\end{align*}
\end{enumerate}
\end{corollary}
\begin{proof} (i)  We evaluate (\ref{eq:main}) with $y=x$ and $i=0$. We have $x^-_y=0$ and $x^0_y=0$, so $E x^+_y = \theta_1 E \hat x$ by Lemma \ref{lem:sum}. 
\\
\noindent (ii) Set $i=1$ in (\ref{eq:main}) and use $x^-_y = \hat y$.
\\
\noindent (iii) Set $i=d$ in (\ref{eq:main}) and use $x^+_y = 0$.
\end{proof}

 \begin{corollary} \label{cor:q111}
 Assume that $\Gamma$ is $Q$-polynomial with respect to $E$.
 Then the Krein parameter $q^1_{1,1}$  satisfies
 \begin{align*}
 q^1_{1,1} = \frac{ \theta_1 \theta^*_1 - \theta_2 \theta^*_0 + \theta_2 - \theta_0 } {\theta_1-\theta_2}.
 \end{align*}
 \end{corollary}
 \begin{proof} Compare 
 Lemma
 \ref{lem:warm} and Corollary  \ref{cor:01D}(i).
 \end{proof}
 
 \noindent The eigenvalue $\theta_2$ appears in the above results.  By \cite[Lemma~19.21]{LSnotes} we find that
 $1+ \theta_1$, $1+ \theta^*_1$ are nonzero and
 \begin{align}
 \frac{1+ \theta_1}{\theta_0-\theta_2} =
 \frac{1+ \theta^*_1}{\theta^*_0-\theta^*_2}.
\label{eq:th2}
 %%%\frac{\theta_0 - \theta_2}{1+ \theta_1} = \frac{\theta^*_0-\theta^*_2}{1+ \theta^*_1}.
 \end{align}

 \section{The Norton product in symmetric form}
 
 We continue to discuss  the distance-regular graph $\Gamma$  and its $Q$-polynomial primitive idempotent
$E$. Pick distinct $x, y \in X$. In the formula 
 (\ref{eq:main}) we computed
  $E \hat x \star E \hat y$. 
We have $E \hat x \star E \hat y=E \hat y \star E \hat x$,
so the right-hand side of   (\ref{eq:main}) must be invariant  if we interchange $x,y$.
In this section, we express the right-hand side of (\ref{eq:main}) in a form that makes this
invariance explicit. We will use a result known as the balanced set condition.

\begin{lemma} \label{lem:bal} {\rm (See \cite[Theorem~1.1]{QPchar}, \cite[Theorem~3.3]{newineq}.)} For distinct $x,y \in X$ we have
\begin{align}
Ex^-_y - E y^-_x &= c_i \frac{\theta^*_1-\theta^*_{i-1}}{\theta^*_0-\theta^*_i} (E \hat x - E \hat y),
\label{eq:bsc1}
\\
 Ex^+_y - E y^+_x &= b_i \frac{\theta^*_1-\theta^*_{i+1}}{\theta^*_0-\theta^*_i} (E \hat x - E \hat y),
 \label{eq:bsc2}
 \end{align}
 where $i=\partial(x,y)$.
 \end{lemma}
 
 \begin{corollary}
 \label{cor:CB}  For distinct $x,y \in X$ we have 
 \begin{align*}
 C(x,y) = C(y,x), \qquad \qquad B(x,y)= B(y,x)
 \end{align*}
 \noindent where
 \begin{align}
 C(x,y) &= 
Ex^-_y -  c_i \frac{\theta^*_1-\theta^*_{i-1}}{\theta^*_0-\theta^*_i} E \hat x,
\label{eq:CC1}
\\
B(x,y) &= Ex^+_y -  b_i \frac{\theta^*_1-\theta^*_{i+1}}{\theta^*_0-\theta^*_i} E \hat x
\label{eq:BB1}
 \end{align}
 and $i = \partial(x,y)$.
\end{corollary}
\begin{proof} Rearrange the terms in (\ref{eq:bsc1}), (\ref{eq:bsc2}).
\end{proof}
\noindent We clarify the meaning of (\ref{eq:CC1}) and (\ref{eq:BB1}).
For $i=1$ we have $C(x,y) = E\hat x + E \hat y$. For $i=d$ we have $B(x,y)=0$.
\medskip

\noindent For distinct $x,y \in X$ we are going to express $E\hat x \star E \hat y$ in terms of $C(x,y)$ and $B(x,y)$.
The following equation will be useful.

 \begin{lemma} \label{cor:cibi}
 We have
 \begin{align*}
 c_i  \frac{ (\theta^*_1 - \theta^*_{i-1})(\theta^*_{i-1} - \theta^*_i)}{\theta^*_0-\theta^*_i}
 + b_i  \frac{ (\theta^*_1 - \theta^*_{i+1})(\theta^*_{i+1} - \theta^*_i)}{\theta^*_0-\theta^*_i}
= (\theta_2 - \theta_1) \theta^*_i + \theta_2 - \theta_0
\end{align*}
for $1 \leq i \leq d-1$ and
 \begin{align*}
 c_d \frac{ (\theta^*_1 - \theta^*_{d-1})(\theta^*_{d-1} - \theta^*_d)}{\theta^*_0-\theta^*_d}
= (\theta_2 - \theta_1) \theta^*_d + \theta_2 - \theta_0.
\end{align*}
\end{lemma}
\begin{proof}  In the equation $0 = E\hat x \star E \hat y - E\hat y \star E \hat x$, expand the right-hand side using
Theorem \ref{thm:main} and evaluate the result using
Lemma \ref{lem:bal}. Examine the outcome using the fact that $E \hat x \not=E \hat y$.
\end{proof}
\noindent The following is our second main result.
\begin{theorem}\label{thm:main2}
Assume that $\Gamma$ is $Q$-polynomial with respect to $E$. Then for distinct $x,y \in X$ we have
\begin{align}
 E{\hat x}\star E{\hat y} &=
 \frac{ (\theta^*_{i-1}-\theta^*_i) C(x,y)+(\theta^*_{i+1}-\theta^*_{i})
B(x,y) +(\theta_2-\theta_0)(E \hat x + E{\hat y})}
  {\vert X \vert (\theta_1-\theta_2)}
  \label{eq:main2}
\end{align}
where $i = \partial (x,y)$. Here  $\theta^*_{d+1} $ denotes an indeterminate.
\end{theorem}
\begin{proof} To verify (\ref{eq:main2}), expand the right-hand side using
(\ref{eq:CC1}), (\ref{eq:BB1}) and evaluate the result using
Theorem \ref{thm:main} along with
Lemma \ref{cor:cibi}.
\end{proof}
\noindent Referring to Theorem \ref{thm:main2},  the formula for $E \hat x \star E \hat y$ can be simplified if $i\in \lbrace 1,d\rbrace$.  This simplification is discussed next.

\begin{corollary} \label{cor:1D} 
Assume that $\Gamma$ is  $Q$-polynomial with respect to $E$. Then {\rm (i), (ii)} hold below.
\begin{enumerate}
\item[\rm (i)] For  $x,y \in X$ at distance $\partial(x,y)=1$,
\begin{align*}
E\hat x \star E \hat y = 
\frac{ (\theta^*_2-\theta^*_1) B(x,y) + (\theta_2-\theta_0+\theta^*_0 - \theta^*_1)(E\hat x + E \hat y)}{\vert X \vert (\theta_1-\theta_2)}.
\end{align*}
\item[\rm (ii)] For $x,y \in X$ at distance $\partial(x,y)=d$,
\begin{align*}
 E{\hat x}\star E{\hat y} &=
 \frac{ (\theta^*_{d-1}-\theta^*_d) C(x,y) +(\theta_2-\theta_0)(E\hat x + E{\hat y})}
  {\vert X \vert (\theta_1-\theta_2)}.
\end{align*}
\end{enumerate}
\end{corollary}
\begin{proof}  (i) Set $i=1$ in (\ref{eq:main2}) and use $C(x,y)= E \hat x + E \hat y$.
\\
\noindent (iii) Set $i=d$ in (\ref{eq:main2}) and use $B(x,y)=0$.
\end{proof}

\section{Remarks}

\noindent 
 The algebraic structure of a $Q$-polynomial distance-regular graph can be described using the concept of a Leonard system \cite[Definition~1.4]{ter2};
 this concept was motivated by a theorem of D.~A.~Leonard  \cite[p.~260]{banIto}, \cite{leonard}.
We refer the reader to \cite{ter2, LSnotes} for the standard notation and basic results about Leonard systems.
The equations below are routinely obtained using the formulas in \cite[Sections~19,~20]{LSnotes}.
Let $\Phi$ denote  any Leonard system with diameter $d\geq 2$. For $\Phi$ we have
\begin{align*}
a^*_1 = \frac{(\theta_1-a_0)\theta^*_1-(\theta_2-a_0)\theta^*_0+(\theta_2-\theta_0) c^*_1}{\theta_1-\theta_2}.
\end{align*}
\noindent If we set $a^*_1=q^1_{1,1}$ and $a_0=0$ and $c^*_1=1$  then we recover the formula in
Corollary
\ref{cor:q111}.
\medskip

\noindent For $\Phi$ we also have
\begin{align*}
\frac{c_1 - a_0 + \theta_1}{\theta_0-\theta_2} =
\frac{c^*_1 - a^*_0 + \theta^*_1}{\theta^*_0-\theta^*_2}.
\end{align*}
\noindent If we set $c_1=1$, $a_0=0$ and $c^*_1=1$, $a^*_0=0$ then we recover 
(\ref{eq:th2}).
\medskip

\noindent For $\Phi$ we also have
\begin{align*} 
 c_i  \frac{ (\theta^*_1 - \theta^*_{i-1})(\theta^*_{i-1} - \theta^*_i)}{\theta^*_0-\theta^*_i}
 + b_i  \frac{ (\theta^*_1 - \theta^*_{i+1})(\theta^*_{i+1} - \theta^*_i)}{\theta^*_0-\theta^*_i}
= (\theta_2 - \theta_1) (\theta^*_i -a^*_0)+ (\theta_2 - \theta_0)c^*_1
\end{align*}
for $1 \leq i \leq d-1$ and
 \begin{align*}
 c_d \frac{ (\theta^*_1 - \theta^*_{d-1})(\theta^*_{d-1} - \theta^*_d)}{\theta^*_0-\theta^*_d}
= (\theta_2 - \theta_1) (\theta^*_d-a^*_0) +(\theta_2 - \theta_0)c^*_1.
\end{align*}
\noindent If we set $a^*_0=0$ and $c^*_1=1$ then we recover the formulas in Lemma \ref{cor:cibi}.
\medskip

\section{Acknowledgement} The author thanks Jia Huang and Kazumasa Nomura for giving this paper a close
reading and offering valuable comments.

\bigskip

\noindent Paul Terwilliger \hfil\break
\noindent Department of Mathematics \hfil\break
\noindent University of Wisconsin \hfil\break
\noindent 480 Lincoln Drive \hfil\break
\noindent Madison, WI 53706-1388 USA \hfil\break
\noindent email: {\tt terwilli@math.wisc.edu }\hfil\break

\end{document}